\documentclass[reqno]{amsart}
\usepackage{amscd}
\usepackage{amsmath,amssymb}
\allowdisplaybreaks[1]

\def\beq{\begin{equation}}
\def\eeq{\end{equation}}
\def\ba{\begin{array}}
\def\ea{\end{array}}
\def\S{\mathbb S}
\def\R{\mathbb R}

\def\N{\mathbb N}

\newcommand{\W}{{\mathcal W}}

\newtheorem{thm}{Theorem}[section]
\newtheorem{lm}[thm]{Lemma}

\theoremstyle{definition}

\theoremstyle{remark}

\begin{document}
\pagestyle{plain}
\title{The singular anisotropic Adams' type inequality in $\mathbb{R}^n$}
\maketitle

\centerline{\scshape Tao Zhang }
\medskip
{\footnotesize
 \centerline{School of Mathematics and Information Science,  Yantai University }
  \centerline{Yantai, China, tzhytu@163.com }}
    \medskip

\centerline{\scshape Meixia Li }
\medskip
{\footnotesize
 \centerline{School of Mathematics and Information Science,  Yantai University }
  \centerline{Yantai, China, lmxdytu@163.com }}
  \medskip

\centerline{\scshape Fan Yang }
\medskip
{\footnotesize
 \centerline{School of Mathematics and Information Science,  Yantai University }
  \centerline{Yantai, China, fyang202404@163.com }}
  \medskip

\centerline{\scshape   Chunqin Zhou}
\medskip
{\footnotesize
 \centerline{ School of Mathematical Sciences, Shanghai Jiao Tong University}
  \centerline{Shanghai, China, cqzhou@sjtu.edu.cn}
  }

\begin{abstract}
In this paper, using anisotropic rearrangement techniques, we first establish the best constants for the singular anisotropic Adams' type inequality with exact growth in $\mathbb{R}^n$.
Furthermore, by the same trick, we also prove the singular anisotropic Adams' type inequality on bounded domain $\Omega\subset \mathbb{R}^n$.
\end{abstract}

\maketitle

\section{Introduction}
Let $\Omega$ be a bounded domain in $\mathbb{R}^n$. By the Sobolev embedding theorem, when $1\leq p<n$, $W_0^{1,p}(\Omega)\subset L^q(\Omega)$, where
$1\leq q\leq\frac{np}{n-p}$. For the limiting situation $p=n$,
$W_0^{1,n}(\Omega)\subset L^q(\Omega)$, $\forall q\geq1$, but $W_0^{1,n}(\Omega) \nsubseteq L^\infty(\Omega)$. Trudinger \cite{T} proved that $W_0^{1,n}(\Omega)\subset L_{\varphi_n}(\Omega)$,
where $L_{\varphi_n}(\Omega)$ is the Orlicz space associated with the Young function
$\varphi_n(t)=e^{\alpha|t|^{\frac{n}{n-1}}}-1$ for some $\alpha>0$.
Further,
in 1971, Moser \cite{JM} proved the classical Moser-Trudinger inequality:
\begin{equation}\label{1.1}
\sup\limits_{u\in W_0^{1,n}(\Omega),\|\nabla u\|_n\leq1}\int_\Omega e^{\alpha|u|^\frac{n}{n-1}}dx\leq C,
\end{equation}
where $\alpha\leq\alpha_n=n\omega_{n-1}^\frac{1}{n-1}$, $\omega_{n-1}$ is the area of the surface of the $n$-dimension unit ball. Moreover, the constant $\alpha_n$ is sharp, i.e. if $\alpha>\alpha_n$, the supremum in (\ref{1.1}) is infinite.

The subcritical Moser-Trudinger inequality in the entire space $\R^n$
was proved by Adachi and Tanaka \cite{AT}:
\begin{equation}\label{1.11}
\sup\limits_{u\in W^{1,n}(\mathbb{R}^n),\|\nabla u\|_n\leq1}\int_{\mathbb{R}^n}\phi(\alpha|u(x)|^\frac{n}{n-1})dx\leq C\|u\|_n^n,
\end{equation}
where $0<\alpha<\alpha_n=n\omega_{n-1}^\frac{1}{n-1}$, $\phi(t)=e^t-\sum\limits_{j=0}^{n-2}\frac{t^j}{j!}$. It is well know that the inequality (\ref{1.11}) is not valid for $\alpha=\alpha_n$.

By strengthening the norm,  if the Dirichlet norm is replaced by the standard
Sobolev norm, namely
$$
\| u \|_{W^{1,n}(\mathbb{R}^n)}=\|\nabla u\|_{L^{n}(\mathbb{R}^n)}+\|u\|_{L^{n}(\mathbb{R}^n)},
$$
Ruf \cite{R} and Li and Ruf \cite{LR} showed the critical Moser-Trudinger type inequality in $\R^2$  and in $\R^n$ respectively as follows
\begin{equation}\label{1.2.0}
\sup\limits_{u\in W^{1,n}(\R^n), \| u \|_{W^{1,n}(\R^n)}\leq 1}\int_{\R^n}\phi(\alpha_n |u(x)|^\frac{n}{n-1})dx\leq C,
\end{equation}
and the constant $\alpha_n$ is sharp. In particular, Lam, Lu and Zhang \cite{LLZ} showed that it is equivalent between the subcritical Moser-Trudinger inequality (\ref{1.11}) with the critical Moser-Trudinger inequality (\ref{1.2.0}).

Later, Ibrahim, Masmoudi and Nakanishi \cite{IMN} established the Moser-Trudinger type inequality with the exact growth in $\R^2$ under Dirichlet norm:
\begin{equation*}
 \int_{\R^2} \frac{e^{4\pi u^2}-1}{(1+|u|)^2}dx\leq C \| u\|_2^2,  \quad \forall u\in W^{1, 2}(\R^2) \text{ with } \|\nabla u\|_2\leq 1.
\end{equation*}

There are many extensions involving Moser-Trudinger type inequalities, we refer to \cite{LT} \cite{LL1} \cite{LL2} \cite{LLZ2} and the references therein.
For the existence of extremal function of the Moser-Trudinger inequalities, we refer to \cite{CC} \cite{F} \cite{Li} \cite{LLZ3}, etc.
These inequalities play important roles in geometric analysis and partial differential equations, see \cite{C} \cite{CY} \cite{FMR}
\cite{O} \cite{LLZ3} \cite{LL3} \cite{LW}.

In 1988, Adams \cite{A} extended the classical Moser-Trudinger inequality to the higher order space $W_0^{m,\frac{n}{m}}(\Omega)$.
Let
\begin{equation*}
\nabla^mu=
\begin{cases}
\nabla \Delta^{\frac{m-1}{2}}u,\text{ when }m\text{ is odd},\\
\Delta^{\frac{m}{2}}u,\text{ when }m\text{ is even}.
\end{cases}
\end{equation*}
 Adams proved that
 there exist a constant $C=C(m,n)$ such that for any $u\in W_0^{m,\frac{n}{m}}(\Omega)$ with $\|\nabla^mu\|_\frac{n}{m}\leq1$, then
\begin{equation}\label{1.2}
\frac{1}{|\Omega|}\int_{\Omega}\exp(\beta|u(x)|^\frac{n}{n-m})dx\leq C,
\end{equation}
for all $\beta\leq\beta(n,m)$, where
\begin{equation*}
\beta(n,m)=
\begin{cases}
\frac{n}{\omega_{n-1}}[\frac{\pi^\frac{n}{2}2^m\Gamma(\frac{m+1}{2})}{\Gamma(\frac{n-m+1}{2})}]^\frac{n}{n-m},\text{ when }m\text{ is odd},\\
\frac{n}{\omega_{n-1}}[\frac{\pi^\frac{n}{2}2^m\Gamma(\frac{m}{2})}{\Gamma(\frac{n-m}{2})}]^\frac{n}{n-m},\text{ when }m\text{ is even}.
\end{cases}
\end{equation*}
Moreover, for any $\beta>\beta(n,m)$, the integral in (\ref{1.2}) can be made as large as possible.

Ruf and Sani \cite{RS}
extended the Adams' inequality (\ref{1.2}) to Sobolev space $W^{m, \frac{n}{m}}(\R^n)$ for any even number $m<n$
by the rearrangement argument.
Lu and Yang \cite{LY} established Adams' inequalities for bi-Laplacian, which is a generalization of
classical Adams' inequality in dimension 4, and showed the existence of extremal functions by the blow-up analysis method.

In 2015, Lu, Tang and Zhu \cite{LTZ} proved Adams' type inequality with exact growth in $\mathbb{R}^{n}$:
\begin{equation}\label{1.3.0}
\int_{\mathbb{R}^n}\frac{\Phi(\beta_n|u|^\frac{n}{n-2})}{(1+|u|)^\frac{n}{n-2}}dx\leq C\|u\|_\frac{n}{2}^\frac{n}{2},
\text{ for all } u\in W^{2,\frac{n}{2}}(\mathbb{R}^n) \text{ and }
\|\Delta u\|_\frac{n}{2}\leq1,
\end{equation}
where $$\Phi(t)=e^t-\sum\limits_{j=0}^{j_\frac{n}{2}-2}\frac{t^j}{j!}, \quad j_\frac{n}{2}=\min\{j\in\mathbb{N}:j\geq\frac{n}{2}\}\geq\frac{n}{2},$$
and $\beta_n=\beta(n,2)=\frac{n}{\omega_{n-1}}[\frac{4\pi^\frac{n}{2}}{\Gamma(\frac{n}{2}-1)}]^\frac{n}{n-2}$.

Lam and Lu \cite{LL4} proved the singular Adams' type inequality with exact growth in $\mathbb{R}^{4}$:
\begin{equation*}
\int_{\mathbb{R}^4}\frac{e^{\alpha u^2}-1}{(1+|u|^{2-\frac{\beta}{2}})|x|^\beta}dx\leq C\|u\|_2^{2-\frac{\beta}{2}},
\text{ for all }u\in W^{2,2}(\mathbb{R}^4), \|\Delta u\|_2\leq1,
\end{equation*}
where $0\leq\beta<4$ and $0<\alpha\leq32\pi^2(1-\frac{\beta}{4})$.

For more details about Adams' type inequalities, we refer to \cite{MS} \cite{CRT} \cite{TC} and the references therein.

Inspired by the previous work on Moser-Trudinger inequality, a lot of researchers devoted themselves to the topic on Moser-Trudinger inequality involving anisotropic norm.
In this paper, we use $F\in C^2(\R^n\backslash \{0\})$
to represent a positive, convex and homogeneous function with its polar $F^o(x)$ standing for a finsler metric on $\R^n$.
The operator $\Delta_F$ is called
Finsler-Laplacian operator, which is defined as
$$
\Delta_F u:=\sum_{i=1}^{n}\frac{\partial}{\partial x_i}(F(\nabla u)F_{\xi_i}(\nabla u)),
$$
where $F_{\xi_i}=\frac{\partial F}{\partial \xi_i}$. In its isotropic case, i.e. $F(\xi)=|\xi|$, $\Delta_F$ is nothing but the ordinary Laplacian.
This operator is closed related to a smooth, convex hypersurface in $\R^n$, which is called
the Wulff shape and was initiated in Wulff' work \cite{W}.
More details about $F(x)$ and $F^o(x)$
will be seen in Section 2.

For a bounded smooth domain $\Omega\subset \R^n$,
by replacing the isotropic Dirichlet norm $\| u \|_{W_0^{1,n}(\Omega)}=(\int_{\Omega}|\nabla u|^ndx)^{\frac{1}{n}}$ by anisotropic norm $(\int_{\Omega}F^n(\nabla u)dx)^{\frac{1}{n}}$ in $W_0^{1,n}(\Omega)$,
Wang and Xia \cite{WX} employed the rearrangement method to prove the following result:
$$
\sup\limits_{u\in W_0^{1,n}(\Omega), \int_{\Omega}F^n(\nabla u)dx\leq 1}\int_{\Omega}e^{\sigma|u|^{\frac{n}{n-1}}}dx\leq C,
$$
where $\sigma\leq \sigma_n=n^{\frac{n}{n-1}}\kappa_n^{\frac{1}{n-1}}$ and
\begin{equation}\label{1.200}
\kappa_n=|{\{x\in \R^n|F^o(x)\leq 1 \}}|
\end{equation}
is  the volume of unit
Wulff ball in $\R^n$. Moreover the constant $\sigma_n$ is sharp.

Liu \cite{Liu} extended Wang and Xia's result to the entire space $\R^n$. He proved the anisotropic Moser-Trudinger type inequality associated with
exact growth in $\R^n$:
$$
\int_{\R^n}\frac{   \phi(\sigma |u|^{\frac{n}{n-1}})   }{   1+|u|^{\frac{n}{n-1}}   }dx\leq C\| u \|_{n}^n, \forall u\in W^{1, n}(\R^n), \text{ with } \int_{\R^n}F^n(\nabla u)dx\leq 1,
$$
where $\phi(t)=e^t-\sum_{j=0}^{n-2}\frac{t^j}{j!}$, $\forall \sigma\leq \sigma_n=n^{\frac{n}{n-1}}\kappa_n^{\frac{1}{n-1}}$.

Recently, Lu, Shen, Xue and Zhu \cite{LSXZ} have established a class of singular anisotropic Moser-Trudinger inequalities by introducing a subtle version of variable transformation. What's more, they also got the existence of extremals for this type of inequality by using a delicate process of blow up analysis.

Although there are some results on anisotropic Moser-Trudinger inequalities, as far as we know, there are fewer results on anisotropic Adams' type inequalities. In \cite{ZYCZ}, Zhang, Yang, Cheng and Zhou
proved the following anisotropic Adams' type inequality with exact growth in $\R^4$:
for any
$u\in W^{2, 2}(\R^4)$ with $\| \Delta_F  u \|_2^2=\int_{\R^{4}}(\Delta_F  u)^2dx\leq 1$,
there exists a constant $C>0$ such that
\begin{equation}\label{1.100}
 \int_{\R^4} \frac{e^{64\kappa_4 u^2}-1}{(1+|u|)^2}dx\leq C \| u\|_2^2,
\end{equation}
and this constant $64\kappa_4$ is sharp. Recalling (\ref{1.200}),
$\kappa_4$ denote the volume of unit
Wulff ball in $\R^4$.
Using the same trick, we also proved
for a bounded domain $\Omega\subset \R^4$, any $u\in W_0^{2, 2}(\Omega)$ with $\| \Delta_F  u \|_2^2=\int_{\Omega}(\Delta_F  u)^2dx\leq 1$,
there exists a constant $C>0$ such that
\begin{equation}\label{1.102}
\int_{\Omega}e^{64\kappa_4 u^2}dx\leq C,
\end{equation}
and this constant $64\kappa_4$ is sharp.

In this paper, we extend the inequalities (\ref{1.100}) and (\ref{1.102}) to its singular version in $\R^{n}$.
As far as we know, this is the first result involving anisotropic Adams' inequality in $\R^{n}$.
The rearrangement technique under finsler metric is key tool in our proof.
As the quasilinear structure of Finsler-Laplacian, the research is more challenging.
Besides, as in the isotropic case, it is an interesting problem involving the existence of extremal function
for anisotropic Adams' inequality. By direct calculation, extremal function of anisotropic Adams' inequality
in Theorem $\ref{main-thm-2}$ satisfy a new and complicated partial differential equations, which
may be studied in future.

 We note that when we consider anisotropic Adams' type inequality in bounded
domain $\Omega\subset \R^{n}$, we denote
$$\| \Delta_F  u \|_\frac{n}{2}=(\int_{\Omega}(\Delta_F  u)^{\frac{n}{2}}dx)^{\frac{2}{n}},$$
and when we consider anisotropic Adams' type inequality in $\R^{n}$, we denote
$$\| \Delta_F  u \|_\frac{n}{2}=(\int_{\R^{n}}(\Delta_F  u)^{\frac{n}{2}}dx)^{\frac{2}{n}}.$$

First, we have the following singular anisotropic Adams' type inequality with exact growth in $\R^n$.

\begin{thm}\label{main-thm-1}
Let $n\geq3$, $0\leq\beta<n$ and $0<\lambda\leq\lambda_n(1-\frac{\beta}{n})$, where
$\lambda_n=n^\frac{n}{n-2}(n-2)^\frac{n}{n-2}\kappa_{n}^\frac{2}{n-2}$.
 Then there exists a constant $C=C(\lambda,\beta,n)>0$ such that for all $u\in W^{2,\frac{n}{2}}(\mathbb{R}^n)$
 with $\|\Delta_F u\|_\frac{n}{2}\leq1$, there holds
\begin{equation*}
\int_{\mathbb{R}^n}\frac{\Phi(\lambda |u|^\frac{n}{n-2})}{(1+|u|^{\frac{n}{n-2}(1-\frac{\beta}{n})})F^o(x)^\beta}dx\leq C\|u\|_\frac{n}{2}^{\frac{n}{2}(1-\frac{\beta}{n})},
\end{equation*}
where $\Phi(t)=e^t-\sum\limits_{j=0}^{j_\frac{n}{2}-2}\frac{t^j}{j!}$, $j_\frac{n}{2}=\min\{j\in\mathbb{\N}:j\geq\frac{n}{2}\}\geq\frac{n}{2}$.
Moreover, the constant $\lambda_n(1-\frac{\beta}{n})$ is sharp and
 the power $\frac{n}{n-2}(1-\frac{\beta}{n})$ in the denominator cannot be replaced with any $q<\frac{n}{n-2}(1-\frac{\beta}{n})$.
\end{thm}

Next using the similar trick, we obtain the singular anisotropic Adams' type inequality in a bounded domain $\Omega\subset \R^{n}$.
\begin{thm}\label{main-thm-2}
Let $n\geq3$, $0\leq\beta<n$ and $0<\lambda\leq\lambda_n(1-\frac{\beta}{n})$, where
$\lambda_n=n^\frac{n}{n-2}(n-2)^\frac{n}{n-2}\kappa_{n}^\frac{2}{n-2}$.
Let $\Omega\subset\mathbb{R}^n$ be a bounded domain. Then there exists a constant $C>0$ such that
\begin{equation*}
\sup\limits_{u\in W^{2,\frac{n}{2}}_0(\Omega),\|\Delta_Fu\|_\frac{n}{2}\leq1}
\int_{\Omega}       \frac{ e^{\lambda |u|^\frac{n}{n-2}} }{ F^o(x)^\beta }      dx\leq C.
\end{equation*}
Moreover, the constant $\lambda_n(1-\frac{\beta}{n})$ is sharp.
\end{thm}
The paper is organized as follows. In Section 2,
we introduce some basic properties of Finsler-Laplacian and rearrangement of functions. In Section 3, we prove some key inequalities which will be useful to the proof of Theorem $\ref{main-thm-1}$. In Section 4, we will prove Theorem $\ref{main-thm-1}$. Finally, Section 5 is devoted to the proof of Theorem $\ref{main-thm-2}$.

\noindent {\bf  Acknowledgements:}
The First author and the fourth author are supported by NSFC of China (No. 12571223).
The authors sincerely appreciate the editors and referees for their careful reading
and helpful comments to improve this paper.
\section{Preliminary}\label{sec2}
In this section, we give some preliminaries on Finsler-Laplacian and rearrangement.

 Throughout the paper, let $F:\mathbb{R}^n\rightarrow\mathbb{R}$ be a convex function of class $C^2(\mathbb{R}^n\setminus\{0\})$
 which is even and positively homogenous degree 1, i.e.
\begin{equation*}
F(t\xi)=|t|F(\xi), \text{ for any } t\in\mathbb{R}, \xi\in\mathbb{R}^n.
\end{equation*}
A typical example is $F(\xi)=(\sum_i|\xi_i|^q)^\frac{1}{q}$ for $q\in[1,\infty)$. We further assume that
\begin{equation*}
F(\xi)>0,\text{ for any }\xi\neq0.
\end{equation*}
By the homogeneity of $F$, there exists two constants $0<a\leq b<\infty$ such that
\begin{equation*}
a|\xi|\leq F(\xi)\leq b|\xi|,\quad \forall \xi\in\mathbb{R}^n.
\end{equation*}
Consider the following minimizing problem
\begin{equation*}
\min\limits_{u\in W_0^{1,n}(\Omega)}\int_{\Omega}F^n(\nabla u)dx,
\end{equation*}
whose Euler-Lagrange equation contains the Finsler-Laplacian operator
\begin{equation*}
\Delta_F u:=\sum\limits^{n}_{i=1}\frac{\partial}{\partial x}(F(\nabla u)F_{\xi_i}(\nabla u)).
\end{equation*}
This operator has been studied by many mathematicians, see \cite{AFTL, BFK, FM, FK, WX} and the references therein.

Consider the map
\begin{equation*}
\phi:\S^{n-1}\rightarrow\mathbb{R}^n, \phi(\xi)=F_\xi(\xi).
\end{equation*}
Its image $\phi(\S^{n-1})$ is a smooth, convex hypersurface in $\mathbb{R}^n$,
which
is called Wulff shape of $F$. If one defines the support function of $F$ as $F^o(x):= \sup\limits_{\xi\in K}\langle x,\xi\rangle$, where $K:=\{x\in\mathbb{R}^n:F(x)\leq1\}$, it is easy to verify that $F^o:\mathbb{R}^n\longmapsto[0,+\infty)$ is also a convex, homogeneous function of class of $C^2(\mathbb{R}^n\setminus\{0\})$. And $F^o$ is dual to $F$ in the sense that
\begin{equation*}
F^o(x)=\sup\limits_{\xi\neq0}\frac{\langle x,\xi\rangle}{F(\xi)},\quad F(x)=\sup\limits_{\xi\neq0}\frac{\langle x, \xi\rangle}{F^o(\xi)}.
\end{equation*}
We denote unit Wulff ball
\begin{equation*}
\mathcal{W}_1:=\{x\in\mathbb{R}^n|F^o(x)\leq1\},
\end{equation*}
and
\begin{equation*}
\kappa_n:=|\mathcal{W}_1|,
\end{equation*}
which is the Lebesgue measure of $\mathcal{W}_1$.
And we denote by $\mathcal{W}_r(x_0)$ the Wulff ball of center at $x_0$ with radius $r$, i.e.,
\begin{equation*}
\mathcal{W}_r(x_0):=\{x\in\mathbb{R}^n|F^o(x-x_0)\leq r\}.
\end{equation*}
For convenience, we denote by $\mathcal{W}_r$ the Wulff ball of center at $0$ with radius $r$.

For later use, we summarize some properties of the function $F$,  which follows directly from the assumption on  $F$,
see \cite{BP, FK, WX2, ZZ}.

\begin{lm}\label{lm-2.1}
We have
\begin{enumerate}
\item[(i)]$|F(x)-F(y)|\leq F(x+y)\leq F(x)+F(y)$;
\item[(ii)]$ \frac{1}{C}\leq|\nabla F(x)|\leq C $, and $ \frac{1}{C}\leq|\nabla F^{o}(x)|\leq C $  for some $C>0$ and any $x\neq 0$;
\item[(iii)]$\langle x,\nabla F(x)\rangle=F(x),\langle x,\nabla F^{o}(x)\rangle=F^{o}(x)$ for any $x\neq 0$;
\item[(iv)]$F(\nabla F^{o}(x))=1$, $F^{o}(\nabla F(x))=1$ for any $x\neq 0$;
\item[(v)]$F^{o}(x) F_{\xi}(\nabla F^{o}(x))=x $ for any $ x\neq 0$;
\item[(vi)]$F_\xi(t\xi)= \text{sgn}(t)F_{\xi}(\xi)$ for any $\xi\neq 0$ and $t\neq 0$.
\end{enumerate}
\end{lm}

Next we give the isoperimetric inequality and co-area formula with respect to $F$.
For a domain  $\Omega\subset \R^n$,  a subset  $E\subset \Omega$  and a function of bounded variation  $u\in BV(\Omega)$, denote the anisotropic bounded variation of $u$ with respect to $F$ by
$$\int_{\Omega}|\nabla u|_{ F}=\sup\{\int_{\Omega}u \text{ div}\sigma dx,  \sigma\in C_{0}^{1}(\Omega; \R^n), F^{o}(\sigma)\leq 1\},$$
and anisotropic perimeter of  $E$  with respect to   $F$ by
$$P_{F}(E):=\int_{\Omega}|\nabla \mathcal{X}_{E}|_{F}dx,$$
where  $\mathcal{X}_{E}$  is the characteristic function of the set  $E$.  It is well known (see also \cite {FM}) that the co-area formula
\begin{equation}\label{2.1}
\int_{\Omega}|\nabla u|_{ F}=\int_{0}^{\infty}P_{F}(|u|>t)dt
\end{equation}
and the isoperimetric inequality
\begin{equation}\label{2.2}
P_{F}(E)\geq n\kappa_{n}^{\frac{1}{n}}|E|^{1-\frac{1}{n}}
\end{equation}
hold. Moreover, the equality in (\ref{2.2}) holds if and only if  $E$  is a Wulff ball.

We denote by $\Omega^\sharp$ the homothetic Wulff ball centered at the origin in $\mathbb{R}^n$, such that
\begin{equation*}
|\Omega|=|\Omega^\sharp|,
\end{equation*}
where $|\cdot|$ indicates the $n$-dimensional Lebesgue measure.
Let $u:\Omega\rightarrow\mathbb{R}$ be a measurable function with real values on $\Omega$.
It is well known that the  distribution function $\mu_u(t):[0,+\infty)\rightarrow[0,+\infty]$
of $u$ is defined as
\begin{equation*}
\mu_u(t)=|x\in\Omega||u(x)|>t|,\text{ for }t\geq0.
\end{equation*}
The decreasing rearrangement $u^*$ of $u$ is defined as
\begin{equation*}
u^*(s):=\sup\{t\geq0|\mu_u(t)>s\},\text{ for }s\geq0,
\end{equation*}
which is the right-continuous, non-increasing function. The support of
$u^*$ satisfies $\text{ supp } u^*\subseteq[0,|\Omega|]$. Furthermore, the maximal function $u^{**}$ of the rearrangement $u^*$ is defined as
\begin{equation*}
u^{**}(s)=\frac{1}{s}\int_0^su^*(t)dt,\text{ for }s\geq0.
\end{equation*}
Since $u^*$ is non-increasing, $u^{**}$ is also non-increasing and $u^*\leq u^{**}$.

By the well known Hardy inequality(see \cite{BS}), we have
\begin{lm}\cite{BS}\label{lm2.2}
If $u\in L^p(\mathbb{R}^n)$ with $1<p<+\infty$ and $\frac{1}{p}+\frac{1}{p^\prime}=1$, then
\begin{equation*}
(\int_0^{+\infty}[u^{**}(s)]^pds)^\frac{1}{p}\leq p^\prime(\int_0^{+\infty}[u^*(s)]^pds)^\frac{1}{p}.
\end{equation*}
In particular, if $suppu\subseteq\Omega$ with $\Omega$ a domain in $\mathbb{R}^n$, then
\begin{equation*}
(\int_0^{|\Omega|}[u^{**}(s)]^pds)^\frac{1}{p}\leq p^\prime(\int_0^{|\Omega|}[u^*(s)]^pds)^\frac{1}{p}.
\end{equation*}
\end{lm}

Finally, the convex symmetrization $u^\sharp$ of $u$ with respect to $F$ is defined by:
\begin{equation*}
u^\sharp(x)=u^*(\kappa_nF^o(x)^n),\text{ for }x\in\Omega^\sharp.
\end{equation*}
In this paper, for convenience, $u^\sharp(x)$ is always written by $u^\sharp(r)$ when $r=F^o(x)$, i.e.
\begin{equation*}
u^\sharp(r)=u^*(\kappa_nr^n),\text{ for }r=F^o(x).
\end{equation*}
\section{Some crucial Lemmas}\label{sec3}
In this section, we give some crucial lemmas which play important roles in the proof of Theorem \ref{main-thm-1} and \ref{main-thm-2}.
First, we have the following lemma which have been proved in \cite{ZYCZ}.

Let $f\in L^2(\Omega)$. We consider the following Dirichlet problem:
\begin{equation}\label{3.4.0}
\left\{
\begin{array}{rlll}
-\Delta_F  u &= & f,\quad  &\text { in } ~~~\Omega,\\
u&=& 0, \quad &\text { on } ~~~\partial\Omega.\\
\end{array}
\right.
\end{equation}

\begin{lm}\cite{ZYCZ}\label{lm3.1}
Let $u\in W_{0}^{1,2}(\Omega)$ be the unique weak solution to (\ref{3.4.0}). Then
$$
u^*(t_1)-u^*(t_2)\leq \frac{1}{(n \kappa_n^{\frac{1}{n}})^2}\int_{t_1}^{t_2}\frac{f^{**}(\xi)}{\xi^{ 1-\frac{2}{n} }}d\xi,\text{ for a.e. } 0<t_1\leq t_2\leq |\Omega|.
$$
Equivalently, if $\Omega^\sharp=\W_R$, then
$$
u^\sharp(R_1)-u^\sharp(R_2)\leq \frac{1}{(n \kappa_n^{\frac{1}{n}})^2}\int_{|\W_{R_1}|}^{|\W_{R_2}|}\frac{f^{**}(\xi)}{\xi^{ 1-\frac{2}{n} }}d\xi,\text{ for a.e. } 0<R_1\leq R_2\leq R.
$$
\end{lm}
\noindent {\bf Remark 1.} In fact, we can prove that an analogue of Lemma \ref{lm3.1} also holds for
functions in $W^{2, 2}(\R^n)$. See more details in \cite{ZYCZ}.

Next, we give a lemma which is important in the proof of the optimal descending growth condition.
\begin{lm}\label{lm3.3}\cite{LTZ}
Given any sequence $a=\{a_k\}_{k\geq0}$.  Let $p>1$, $\|a\|_1=\sum_{k=0}^{\infty}|a_k|$, $\|a\|_p=(\sum_{k=0}^{\infty}|a_k|^p)^{\frac{1}{p}}$, $\|a\|_{(e)}=(\sum_{k=0}^{\infty}|a_k|^pe^k)^{\frac{1}{p}}$ and $\mu(h)=\inf\{\|a\|_{(e)}:\|a\|_1=h,\|a\|_p\leq1\}$. Then for $h>1$, we have
\begin{equation*}
\mu(h)\thicksim\frac{e^\frac{h^\frac{p}{p-1}}{p}}{h^\frac{1}{p-1}}.
\end{equation*}
\end{lm}
By Lemma \ref{lm3.3}, we can obtain the optimal descending growth condition, which can be seen as the exponential version of the radial Sobolev inequality.
\begin{lm}\label{lm3.2}
Let $u\in W^{2,\frac{n}{2}}(\mathbb{R}^n)$ and $R>0$. If $u^\sharp(R)>1$ and $f:=-\Delta_Fu$ in $\mathbb{R}^n$ satisfies
\begin{equation*}
\int_{|\W_R|}^{+\infty}[f^{\ast\ast}(s)]^\frac{n}{2}ds\leq\left(\frac{n}{n-2}\right)^\frac{n}{2},
\end{equation*}
then we have
\begin{equation*}
\frac{R^ne^{\kappa_n^{\frac{2}{n-2}}[n(n-2)u^\sharp(R)]^\frac{n}{n-2}}}{[u^\sharp(R)]^\frac{n}{n-2}}\leq C\|u^\sharp\|_{L^\frac{n}{2}(\mathbb{R}^n\setminus \mathcal{W}_R)}^\frac{n}{2},
\end{equation*}
where $C>0$ is a constant independent of $u$ and $R$.
\end{lm}
\begin{proof}
Let $h_k=c_nu^*(\kappa_n R^ne^k)$, where $c_n=\frac{n-2}{n}
[n\kappa_n^{\frac{1}{n}}]^2$. Define $a_k=h_k-h_{k+1}$ and $a=\{a_k\}_{k\geq0}$.
Then $a_k\geq 0$ and $\sum_{k=0}^{+\infty}|a_k|=h_0=c_nu^*(\kappa_n R^n)$.

By Lemma \ref{lm3.1} and H\"{o}lder inequality, we have
\begin{align*}
h_k-h_{k+1}&=c_n[u^*(\kappa_n R^ne^k)-u^*(\kappa_n R^ne^{k+1})]\\
&\leq\frac{c_n}{[n\kappa_n^\frac{1}{n}]^2}\int_{\kappa_n R^ne^k}^{\kappa_n R^ne^{k+1}}\frac{f^{\ast\ast}(s)}{s^{1-\frac{2}{n}}}ds\\
&\leq \frac{c_n}{[n\kappa_n^{\frac{1}{n}}]^2}
\left(\int_{\kappa_n R^ne^k}^{\kappa_n R^ne^{k+1}}|f^{**}(s)|^\frac{n}{2}ds\right)^\frac{2}{n}\left(\int_{\kappa_n R^ne^k}^{\kappa_n R^ne^{k+1}}\left(\frac{1}{s^{1-\frac{2}{n}}}\right)^\frac{n}{n-2}ds\right)^{1-\frac{2}{n}}\\
&\leq\frac{c_n}{[n\kappa_n^{\frac{1}{n}}]^2}
\left(\int_{\kappa_n R^ne^k}^{\kappa_n R^ne^{k+1}}|f^{**}(s)|^\frac{n}{2}ds\right)^\frac{2}{n}
\end{align*}
Therefore we obtain
\begin{equation*}
\|a\|_\frac{n}{2}=(\sum_{k=0}^{\infty}|a_k|^\frac{n}{2})^\frac{2}{n}
=(\sum_{k=0}^{\infty}|h_k-h_{k+1}|^\frac{n}{2})^\frac{2}{n}\leq\frac{n-2}{n}
(\int_{|\mathcal{W}_R|}^{+\infty}|f^{**}(s)|^\frac{n}{2}ds)^\frac{2}{n}\leq1.
\end{equation*}
On the other hand, noticing that
\begin{align*}
&\frac{\int_R^{+\infty}|u^*(\kappa_nr^n)|^\frac{n}{2}r^{n-1}dr}{R^n}\\
=&\frac{\sum\limits_{k=0}^{\infty}\int_{Re^\frac{k}{n}}^{Re^\frac{k+1}{n}}|u^*(\kappa_nr^n)|^\frac{n}{2}r^{n-1}dr}{R^n}\\
\geq&\sum\limits_{k=0}^{\infty}\frac{[u^*(\kappa_n R^ne^{k+1})]^\frac{n}{2}}{R^n}\int_{Re^\frac{k}{n}}^{Re^\frac{k+1}{n}}r^{n-1}dr\\
\geq& C\sum\limits_{k=0}^{\infty}[u^*(\kappa_n R^ne^{k+1})]^\frac{n}{2}e^{k+1}\\
\geq& C\sum\limits_{k=1}^{\infty}(h_k)^\frac{n}{2}e^k\geq C\sum\limits_{k=1}^{\infty}(a_k)^\frac{n}{2}e^k.
\end{align*}
hence we have
\begin{equation}\label{3.1}
\|a\|_{(e)}^\frac{n}{2}=a_0^\frac{n}{2}+\sum\limits_{k=1}^{\infty}(a_k)^\frac{n}{2}e^k\leq C [h_0^\frac{n}{2}+\frac{\int_R^{+\infty}(u^\sharp(r))^\frac{n}{2}r^{n-1}dr}{R^n}].
\end{equation}
Next, we estimate $h_0$. Set $R<r<Re^{\frac{b}{n}}$,
where $b=((\frac{n}{n-2})^{-\frac{n}{2}}\frac{(n\kappa_n^\frac{1}{n})^2}{2})^\frac{n}{n-2}$. Since
\begin{align*}
&h_0-c_nu^*(\kappa_nr^n)\\
\leq&\frac{c_n}{[n\kappa_n^\frac{1}{n}]^2}\int_{|\mathcal{W}_R|}^{|\mathcal{W}_r|}\frac{f^{**}(s)}{s^{1-\frac{2}{n}}}ds\\
\leq&\frac{c_n}{[n\kappa_n^\frac{1}{n}]^2}[\int_{|\mathcal{W}_R|}^{|\mathcal{W}_r|}[f^{**}(s)]^\frac{n}{2}ds]^\frac{2}{n}b^\frac{n-2}{n}\\
\leq&\frac{c_n}{2}\leq\frac{c_nu^{\sharp}(R)}{2}=\frac{h_0}{2}.
\end{align*}
we have $h_0\leq C u^*(\kappa_n r^n)$  when $R<r<Re^{\frac{b}{n}}$. Hence we obtain
\begin{equation}\label{3.2}
\frac{\int_R^\infty(u^\sharp(r))^\frac{n}{2}r^{n-1}dr}{R^n}\geq C\frac{\int_R^{Re^{\frac{b}{n}}}h_0^\frac{n}{2}r^{n-1}dr}{R^n}\geq C
 h_0^\frac{n}{2}.
\end{equation}
Combining inequalities (\ref{3.1}) and (\ref{3.2}), we have
\begin{equation*}
\|a\|_{(e)}^\frac{n}{2}\leq C\frac{\int_R^\infty(u^\sharp(r))^\frac{n}{2}r^{n-1}dr}{R^n}.
\end{equation*}
Let $\alpha=(c_n)^\frac{n}{n-2}$. By using Lemma \ref{lm3.3}, we have
\begin{equation*}
\|a\|_{(e)}^\frac{n}{2}\geq C\frac{\exp(h_0^\frac{n}{n-2})}{(h_0)^{\frac{n}{n-2}}}\geq C
\frac{\exp[(c_nu^\sharp(R))^\frac{n}{n-2}]}
{(u^\sharp(R))^\frac{n}{n-2}}=\frac{\exp[\alpha(u^\sharp(R))^\frac{n}{n-2}]}{[u^\sharp(R)]^\frac{n}{n-2}}.
\end{equation*}
Thus we completed the proof of the lemma.
\end{proof}
We now give the Hardy-Littlewood type inequality under finsler metric due to Liu \cite{Liu2}.
\begin{lm}\label{lm3.4}\cite{Liu2}
Let $\varphi$ and $h$ be measurable functions in $\R^n$.
Assume that $H$ is a increasing function. Then we have
\begin{equation*}
\int_{\mathbb{R}^n}h(x)H(\varphi)dx\leq\int_{\mathbb{R}^n}h^\sharp(x)H(\varphi^{\sharp})dx.
\end{equation*}
\end{lm}
By Lemma \ref{lm3.4}, we have the following fact.
\begin{lm}\label{lm3.5}
Let $0\leq\beta<n$ and $\lambda>0$, we have
\begin{equation*}
\int_{\mathbb{R}^n}\frac{\Phi(\lambda |u|^\frac{n}{n-2})dx}{(1+|u|^{\frac{n}{n-2}(1-\frac{\beta}{n})})F^o(x)^\beta}\leq\int_{\mathbb{R}^n}\frac{\Phi(\lambda (u^\sharp)^\frac{n}{n-2})dx}{(1+(u^\sharp)^{\frac{n}{n-2}(1-\frac{\beta}{n})})F^o(x)^\beta},
\end{equation*}
for any $u\in W^{2,\frac{n}{2}}(\mathbb{R}^n)$.
\end{lm}
\begin{proof}
Let $u\in W^{2,\frac{n}{2}}(\mathbb{R}^n)$ and set
\begin{equation*}
(H\circ|u|)(x)=\frac{\Phi(\lambda |u|^\frac{n}{n-2})}{1+|u|^{\frac{n}{n-2}(1-\frac{\beta}{n})}}\quad\text{and}\quad g(x)=\frac{1}{F^o(x)^\beta},
\end{equation*}
where
\begin{equation*}
H(t)=\frac{\Phi(\lambda t^\frac{n}{n-2})}{1+t^{\frac{n}{n-2}(1-\frac{\beta}{n})}}.
\end{equation*}
First, we note that
\begin{equation*}
g(x)=\frac{1}{F^o(x)^\beta}=g^\sharp(x).
\end{equation*}
Then by Lemma \ref{lm3.4}, it is sufficient to
claim that $H(t)$ is nondecreasing on $\mathbb{R}^+$. In fact, let us denote $\xi(t)=\lambda t^\frac{n}{n-2}$, and by direct calculation we have
\begin{align*}
&H^\prime(t)\\
=&\frac{\frac{n}{n-2}\lambda t^\frac{2}{n-2}\bigl(\Phi(\xi)+\frac{(\lambda t^\frac{n}{n-2})^{j_\frac{n}{2}-2}}{(j_\frac{n}{2}-2)!}\bigr)(1+t^{\frac{n}{n-2}
(1-\frac{\beta}{n})})-\frac{n}{n-2}(1-\frac{\beta}{n})t^\frac{2-\beta}{n-2}\Phi(\xi)}
{(1+t^{\frac{n}{n-2}(1-\frac{\beta}{n})})^2}\\
=&\frac{(\Phi(\xi)+\frac{\xi^{j_\frac{n}{2}-2}}{(j_\frac{n}{2}-2)!})[\frac{n}{n-2}\xi+\frac{n}{n-2}\lambda t^\frac{n+2-\beta}{n-2}-\frac{n}{n-2}(1-\frac{\beta}{n})t^\frac{2-\beta}{n-2}]+\frac{n}{n-2}(1-\frac{\beta}{n})
\frac{\xi^{j_\frac{n}{2}-2}}{(j_\frac{n}{2}-2)!}t^\frac{2-\beta}{n-2}}{(1+t^{\frac{n}{n-2}(1-\frac{\beta}{n})})^2}\\
=&\frac{\frac{n}{n-2}t^\frac{2-\beta}{n-2}}{(1+t^{\frac{n}{n-2}(1-\frac{\beta}{n})})^2}[(\Phi(\xi)+
\frac{\xi^{j_\frac{n}{2}-2}}{(j_\frac{n}{2}-2)!})(\lambda t^\frac{\beta}{n-2}+\lambda t^\frac{n}{n-2}-(1-\frac{\beta}{n}))+(1-\frac{\beta}{n})\frac{\xi^{j_\frac{n}{2}-2}}{(j_\frac{n}{2}-2)!}]\\
=&\frac{\frac{n}{n-2}t^\frac{2-\beta}{n-2}h(t)}{(1+t^{\frac{n}{n-2}(1-\frac{\beta}{n})})^2},
\end{align*}
where
\begin{equation*}
h(t)=(\Phi(\xi)+\frac{(\lambda t^\frac{n}{n-2})^{j_\frac{n}{2}-2}}{(j_\frac{n}{2}-2)!})(\lambda t^\frac{\beta}{n-2}+\lambda t^\frac{n}{n-2}-(1-\frac{\beta}{n}))+(1-\frac{\beta}{n})\frac{(\lambda t^\frac{n}{n-2})^{j_\frac{n}{2}-2}}{(j_\frac{n}{2}-2)!}.
\end{equation*}
Now we distinguish
two cases to calculate the derivative of $h(t)$.

\textbf{Case 1.} $n=3$ or $n=4$.

In this case, since
\begin{equation*}
h(t)=e^{\lambda t^\frac{n}{n-2}}
(\lambda t^\frac{\beta}{n-2}+\lambda t^\frac{n}{n-2}-(1-\frac{\beta}{n}))+(1-\frac{\beta}{n})\lambda t^\frac{n}{n-2}.
\end{equation*}
Hence we have
\begin{align*}
&h^\prime(t)\\
=&\frac{n}{n-2}\lambda t^\frac{2}{n-2}
e^{\lambda t^\frac{n}{n-2}}
(\lambda t^\frac{\beta}{n-2}+\lambda t^\frac{n}{n-2}-(1-\frac{\beta}{n}))\\
+&e^{\lambda t^\frac{n}{n-2}}
(\frac{\beta}{n-2}\lambda t^\frac{2+\beta-n}{n-2}+\frac{n}{n-2}\lambda t^\frac{2}{n-2})+\lambda\frac{n}{n-2}(1-\frac{\beta}{n})
t^\frac{2}{n-2}\\
=&\frac{n}{n-2}\lambda t^\frac{2}{n-2}
e^{\lambda t^\frac{n}{n-2}}
(\lambda t^\frac{\beta}{n-2}+\lambda t^\frac{n}{n-2}+\frac{\beta}{n}t^\frac{\beta-n}{n-2}+\frac{\beta}{n})
+\lambda\frac{n}{n-2}(1-\frac{\beta}{n})t^\frac{2}{n-2}
\\
\geq&0.
\end{align*}

\textbf{Case 2.} $n\geq5$.

In this case, we have
\begin{align*}
&h^\prime(t)\\
=&\frac{n}{n-2}\lambda t^\frac{2}{n-2}[\Phi(\xi)+\frac{\xi^{j_\frac{n}{2}-2}}{(j_\frac{n}{2}-2)!}+\frac{\xi^{j_\frac{n}{2}-3}}{(j_\frac{n}{2}-3)!}](\lambda t^\frac{\beta}{n-2}+\lambda t^\frac{n}{n-2}-(1-\frac{\beta}{n}))\\
+&(\Phi(\xi)+\frac{\xi^{j_\frac{n}{2}-2}}{(j_\frac{n}{2}-2)!})(\frac{\beta}{n-2}\lambda t^\frac{2+\beta-n}{n-2}+\frac{n}{n-2}\lambda t^\frac{2}{n-2})+\lambda\frac{n}{n-2}(1-\frac{\beta}{n})\frac{\xi^{j_\frac{n}{2}-3}}{(j_\frac{n}{2}-3)!}t^\frac{2}{n-2}\\
=&\frac{n}{n-2}\lambda t^\frac{2}{n-2}[\Phi(\xi)+\frac{(\lambda t^\frac{n}{n-2})^{j_\frac{n}{2}-2}}{(j_\frac{n}{2}-2)!}](\lambda t^\frac{\beta}{n-2}+\lambda t^\frac{n}{n-2}+\frac{\beta}{n}t^\frac{\beta-n}{n-2}+\frac{\beta}{n})\\
+&\frac{n}{n-2}\lambda t^\frac{2}{n-2} \frac{\xi^{j_\frac{n}{2}-3}}{(j_\frac{n}{2}-3)!}(\lambda t^\frac{\beta}{n-2}+\lambda t^\frac{n}{n-2})
\\
\geq&0.
\end{align*}
Therefore, we obtain
\begin{equation*}
h(t)\geq h(0)=0,\text{ for } t\geq0.
\end{equation*}
Hence we get that  $H^\prime(t)\geq0$ when $t\geq0$, which means that $H$ is nondecreasing on $\mathbb{R}^+$.
Thus we completed the proof of the lemma.

\end{proof}

\section{Proof of Theorem \ref{main-thm-1}}\label{sec5}
In this section, we will prove Theorem \ref{main-thm-1}.
First, we recall the following inequality \cite{LL5}, which will be used in the proof of Theorem.
\begin{lm}\label{lm4.1}\cite{LL5}
Let $0<\gamma\leq1$, $1<p<\infty$ and $a(s,t)$ be a nonnegative measurable function on $(-\infty,+\infty)\times[0,+\infty)$ such that
\begin{equation*}
a(s,t)\leq1\text{ for }0<s<t,
\end{equation*}
\begin{equation*}
\sup\limits_{t>0}(\int_{-\infty}^0+\int_t^{\infty}a(s,t)^{p^\prime}ds)^{\frac{1}{p^\prime}}=b<\infty.
\end{equation*}
Then there is a constant $c_0=c_0(p,b,\gamma)$ such that if for $\phi\geq0$,
\begin{equation*}
\int_{-\infty}^{\infty}\phi(s)^pds\leq1,
\end{equation*}
then
\begin{equation*}
\int_0^{\infty}e^{-F_\gamma(t)}dt\leq c_0,
\end{equation*}
where
\begin{equation*}
F_\gamma(t)=\gamma t-\gamma(\int_{-\infty}^{\infty}a(s,t)\phi(s)ds)^{p^\prime}.
\end{equation*}
\end{lm}
Now we give the proof of Theorem $\ref{main-thm-1}$.

{\bf Proof of Theorem \ref{main-thm-1}}:
Let us fix $u\in W^{2, \frac{n}{2}}(\R^n)$ such that $\|\Delta_F u\|_\frac{n}{2}\leq1$ and define
\begin{equation*}
R_{0}=\inf \left\{r>0:u^\sharp(r)\leq1\right\}\in[0,\infty).
\end{equation*}
Without loss of generality, we may assume that $R_{0}>0$.

By Lemma \ref{lm3.5}, we have
\begin{align*}
&\int\limits_{\mathbb{R}^{n}} \frac{\Phi(\lambda_n (1-\frac{\beta}{n}) |u|^\frac{n}{n-2})}{\left(1 + |u|^{\frac{n}{n-2}\left(1-\frac{\beta}{n}\right)}\right)F^o(x)^{\beta}}dx\\
\leq&\int\limits_{\mathbb{R}^{n}} \frac{\Phi(\lambda_n (1-\frac{\beta}{n}) (u^{\sharp})^\frac{n}{n-2})}{\left(1 + (u^{\sharp})^{\frac{n}{n-2}\left(1-\frac{\beta}{n}\right)}\right)F^o(x)^{\beta}}dx \\
=&\int\limits_{\mathcal{W}_{R_{0}}}\frac{\Phi(\lambda_n (1-\frac{\beta}{n})(u^{\sharp})^\frac{n}{n-2})}{\left(1 + (u^{\sharp})^{\frac{n}{n-2}\left(1-\frac{\beta}{n}\right)}\right)F^o(x)^{\beta}}dx\\
+&\int\limits_{\mathbb{R}^{n}\setminus \mathcal{W}_{R_{0}}}\frac{\Phi(\lambda_n (1-\frac{\beta}{n})(u^{\sharp})^\frac{n}{n-2})}{\left(1 + (u^{\sharp})^{\frac{n}{n-2}\left(1-\frac{\beta}{n}\right)}\right)F^o(x)^{\beta}}dx\\
=&I+J.
\end{align*}
We now estimate $J$. Since when $0<t\leq1$, $\Phi(\lambda_n(1-\frac{\beta}{n}) t^\frac{n}{n-2})\leq C t^\frac{n}{2}$,
and $u^{\sharp}(r)\leq1$ on $\mathbb{R}^{n}\setminus \mathcal{W}_{R_{0}}$ , we obtain
\begin{align*}
J &\leq C \int\limits_{\{u^{\sharp} \leq 1\}} \frac{(u^{\sharp})^\frac{n}{2}}{F^o(x)^{\beta}}dx\\
&\leq C \int\limits_{\{u^{\sharp} \leq 1; F^o(x) \geq \|u^{\sharp}\|_\frac{n}{2}^\frac{1}{2}\}} \frac{(u^{\sharp})^\frac{n}{2}}{F^o(x)^{\beta}}dx + C \int\limits_{\{u^{\sharp} \leq 1; F^o(x)< \|u^{\sharp}\|_\frac{n}{2}^\frac{1}{2}\}} \frac{(u^{\sharp})^\frac{n}{2}}{F^o(x)^{\beta}}dx\\
&\leq C \left[ \int\limits_{\mathbb{R}^{n}}(u^{\sharp})^\frac{n}{2}dx \right]^{1 - \frac{\beta}{n}}\\
&= C \|u\|_\frac{n}{2}^{\frac{n}{2}\left(1-\frac{\beta}{n}\right)}.
\end{align*}
Next, we estimate $I$.  Set
\begin{align*}
f&=-\Delta_Fu,\\
\alpha&=\int_0^{+\infty}\left[ f^{\ast\ast}(s) \right]^\frac{n}{2} ds.
\end{align*}
Using Lemma \ref{lm2.2} with $p=\frac{n}{2}$ and the assumption $\|\Delta_F u\|_{\frac{n}{2}}\leq 1$,
we have $\alpha\leq\left(\frac{n}{n-2}\right)^\frac{n}{2}$.

Now we fix $\varepsilon_{0}\in(0,1)$ (we note that $\varepsilon_{0}$ is independent of $u$), and choose $R_{1}=
R_{1}(u)>0$ such that
\begin{align}
\int_{0}^{|\mathcal{W}_{R_{1}}|} \left[ f^{\ast\ast}(s) \right]^{\frac{n}{2}} ds &= \alpha (1 - \varepsilon_{0}),\label{5.1}\\
\int_{|\mathcal{W}_{R_{1}}|}^{\infty} \left[ f^{\ast\ast}(s) \right]^{\frac{n}{2}} ds &= \alpha \varepsilon_{0}.\label{5.2}
\end{align}
Applying Lemma \ref{lm3.1} (see also Remark 1) and H\"{o}lder inequality, we
have for $0<r_{1}<r_{2}$,
\begin{equation}\label{5.3}
u^{\sharp}(r_{1}) - u^{\sharp}(r_{2}) \leq \frac{1}{[n\kappa_n^\frac{1}{n}]^2}
\left(\int_{|\mathcal{W}_{r_{1}}|}^{|\mathcal{W}_{r_{2}}|} \left[ f^{\ast\ast}(s) \right]^\frac{n}{2} ds\right)^\frac{2}{n}\left(\ln\frac{r_{2}^{n}}{r_{1}^{n}}\right)^{1-\frac{2}{n}}.
\end{equation}
Hence, by (\ref{5.1}), (\ref{5.2}) and (\ref{5.3}), we obtain
\begin{align}
u^{\sharp}(r_{1}) - u^{\sharp}(r_{2}) &\leq \frac{1}{[n\kappa_n^\frac{1}{n}]^2}\left(\alpha(1-\varepsilon_0)\right)^\frac{2}{n}
\left(\ln\frac{r_{2}^{n}}{r_{1}^{n}}\right)^{1-\frac{2}{n}}\quad \text{for }0<r_{1}<r_{2}\leq R_{1},\label{5.4}\\
u^{\sharp}(r_{1}) - u^{\sharp}(r_{2}) &\leq \frac{1}{[n\kappa_n^\frac{1}{n}]^2}\left(\alpha\varepsilon_0\right)^\frac{2}{n}
\left(\ln\frac{r_{2}^{n}}{r_{1}^{n}}\right)^{1-\frac{2}{n}}\quad \text{for }R_{1}\leq r_{1}<r_{2}.\label{5.5}
\end{align}
We distinguish two cases.\\
\textbf{Case 1.} $0<R_{0}\leq R_{1}$.

In this case, by (\ref{5.4}), for $0<r\leq R_{0}$, we have
\begin{equation*}
u^{\sharp}(r)\leq1+\frac{1}{[n\kappa_n^\frac{1}{n}]^2}(\alpha(1-\varepsilon_0))^\frac{2}{n}\left(\ln\frac{R_0^n}{r^n}\right)^{1-\frac{2}{n}}.
\end{equation*}
It is well known that there exists a constant $C_\varepsilon=(1-(1+\varepsilon)^{-\frac{n-2}{2}})^\frac{-2}{n-2}$ such that
\begin{equation}\label{5.100}
(1+s^\frac{n-2}{n})^\frac{n}{n-2}\leq(1+\varepsilon)s+C_\varepsilon,\text{ for }s>0.
\end{equation}
Then
\begin{equation*}
[u^\sharp(r)]^\frac{n}{n-2}\leq(1+\varepsilon)\frac{(\alpha(1-\varepsilon_0))^\frac{2}{n-2}}
{[n\kappa_n^\frac{1}{n}]^\frac{2n}{n-2}}\ln(\frac{R_0}{r})^n+C_\varepsilon.
\end{equation*}
Set $\varepsilon=1-(1-\varepsilon_0)^\frac{2}{n-2}$, then $(1+\varepsilon)(1-\varepsilon_0)^\frac{2}{n-2}<1$.
Since $\frac{\lambda_n}{[n\kappa_n^\frac{1}{n}]^\frac{2n}{n-2}}=(1-\frac{2}{n})^{\frac{n}{n-2}}$ and
$\alpha\leq (\frac{n}{n-2})^{\frac{n}{2}}$,
we have
\begin{align}\label{5.1000}
I= &\int\limits_{\mathcal{W}_{R_{0}}} \frac{\Phi(\lambda_n (1-\frac{\beta}{n})(u^\sharp)^\frac{n}{n-2})}{\left( 1 + (u^{\sharp})^{\frac{n}{n-2}\left(1-\frac{\beta}{n}\right)} \right) F^o(x)^{\beta}}dx \nonumber\\
\leq&\int_{\mathcal{W}_{R_0}}\exp(\lambda_n(u^\sharp)^\frac{n}{n-2})\frac{1}{F^o(x)^\beta}dx\nonumber\\
=&n\kappa_n\int_0^{R_0}\exp(\lambda_n[u^*(\kappa_nr^n)]^\frac{n}{n-2})r^{n-1-\beta}dr\nonumber\\
\leq& C\int_0^{R_0}\exp\biggl(\lambda_n\bigl((1+\varepsilon)\frac{(\alpha(1-\varepsilon_0))^\frac{2}{n-2}}{[n\kappa_n^\frac{1}{n}]^\frac{2n}{n-2}}
\ln(\frac{R_0}{r})^n+C_\varepsilon\bigr)\biggr)r^{n-1-\beta}dr\nonumber\\
\leq&C\int_0^{R_0}R_0^{n(1-\varepsilon^2)}r^{n-1-\beta-n(1-\varepsilon^2)}dr\\
\leq&C R_0^{n-\beta}\nonumber\\
=&C\left(\int_{\mathcal{W}_{R_0}}dx\right)^{1-\frac{\beta}{n}}\nonumber\\
\leq&C\left(\int_{\mathcal{W}_{R_0}}(u^\sharp)^\frac{n}{2}dx\right)^{1-\frac{\beta}{n}}\nonumber\\
=&C\|u^\sharp\|_\frac{n}{2}^{\frac{n}{2}\left(1-\frac{\beta}{n}\right)}\nonumber\\
=&C\|u\|_\frac{n}{2}^{\frac{n}{2}\left(1-\frac{\beta}{n}\right)}\nonumber.
\end{align}
Thus we get the desired inequality when $R_{0}\leq R_{1}$.
We note that in (\ref{5.1000}), we can choose $\epsilon_0$ is sufficiently small, which
asserts that $\epsilon$ is sufficiently close to $1$ and
then $r^{n-1-\beta-n(1-\varepsilon^2)}$ is integrable.
\\
\textbf{Case 2.} $0<R_{1}<R_{0}$.

In this case, we write that
\begin{align*}
I &= \int_{\mathcal{W}_{R_{0}}} \frac{\Phi(\lambda_n (1-\frac{\beta}{n})(u^\sharp)^\frac{n}{n-2})}{\left( 1 + (u^{\sharp})^{\frac{n}{n-2}\left(1-\frac{\beta}{n}\right)} \right) F^o(x)^{\beta}} \, dx \\
&=\int\limits_{\mathcal{W}_{R_{1}}}\frac{\Phi(\lambda_n (1-\frac{\beta}{n})(u^\sharp)^\frac{n}{n-2})}{\left( 1 + (u^{\sharp})^{\frac{n}{n-2}\left(1-\frac{\beta}{n}\right)} \right) F^o(x)^{\beta}}
+\int\limits_{\mathcal{W}_{R_{0}}\setminus \mathcal{W}_{R_{1}}} \frac{\Phi(\lambda_n (1-\frac{\beta}{n})(u^\sharp)^\frac{n}{n-2})}{\left( 1 + (u^{\sharp})^{\frac{n}{n-2}\left(1-\frac{\beta}{n}\right)} \right) F^o(x)^{\beta}} \, dx \\
&=I_{1}+I_{2}.
\end{align*}
Again, by (\ref{5.5}), we have for $R_{1}\leq r\leq R_{0}$,
\begin{equation*}
u^\sharp(r)\leq1+\frac{1}{[n\kappa_n^\frac{1}{n}]^2}(\alpha\varepsilon_0)^\frac{2}{n}(\ln(\frac{R_0}{r})^n)^{1-\frac{2}{n}},
\end{equation*}
and
\begin{equation*}
[u^\sharp(r)]^\frac{n}{n-2}\leq(1+\varepsilon_1)\frac{[\alpha\varepsilon_0]^\frac{2}{n-2}}{[n\kappa_n^\frac{1}{n}]^\frac{2n}{n-2}}\ln(\frac{R_0}{r})^n+C_{\varepsilon_1},
\end{equation*}
where $\varepsilon_1=1-\varepsilon_0^\frac{2}{n-2}$ and $C_{\varepsilon_1}=\left(1-(1+\varepsilon_1)^{-\frac{n-2}{2}}\right)^{-\frac{2}{n-2}}$. Hence we have
\begin{align*}
I_{2} =&\int\limits_{\mathcal{W}_{R_{0}} \setminus \mathcal{W}_{R_{1}}} \frac{\Phi(\lambda_n (1-\frac{\beta}{n})(u^{\sharp})^\frac{n}{n-2})}{\left( 1 + (u^{\sharp})^{\frac{n}{n-2}\left(1-\frac{\beta}{n}\right)} \right) F^o(x)^{\beta}}dx \\
\leq&\int\limits_{\mathcal{W}_{R_0} \setminus \mathcal{W}_{R_1}}\exp(\lambda_n(u^\sharp)^\frac{n}{n-2})\frac{1}{F^o(x)^\beta}dx\\
=&n\kappa_n\int_{R_1}^{R_0}\exp(\lambda_n[u^*(\kappa_nr^n)]^\frac{n}{n-2})r^{n-1-\beta}dr\\
\leq& n\kappa_n\int_{R_1}^{R_0}\exp\biggl(\lambda_n\bigl((1+\varepsilon_1)\frac{(\alpha\varepsilon_0)^\frac{2}{n-2}}{[n\kappa_n^\frac{1}{n}]^\frac{2n}{n-2}}
\ln(\frac{R_0}{r})^n+C_{\varepsilon_1}\bigr)\biggr)r^{n-1-\beta}dr\\
\leq& CR_0^{n(1-\varepsilon_1^2)}\int_{R_{1}}^{R_{0}}r^{n-1-\beta-n(1-\varepsilon_1^2)}dr \\
=& CR_0^{n(1-\varepsilon_1^2)}\left[R_{0}^{n-\beta-n(1-\varepsilon_1^2)}-R_{1}^{n-\beta-n(1-\varepsilon_1^2)}\right] \\
=& C\left[ R_{0}^{n-\beta} - R_0^{n(1-\varepsilon_1^2)}R_{1}^{n-\beta-n(1-\varepsilon_1^2)} \right] \\
\leq& C  R_{0}^{n-\beta}\\
=&C\left(\int_{\mathcal{W}_{R_0}}dx\right)^{1-\frac{\beta}{n}}\\
\leq&C\left(\int_{\mathcal{W}_{R_0}}(u^\sharp)^\frac{n}{2}dx\right)^{1-\frac{\beta}{n}}\\
=&C\|u^\sharp\|_\frac{n}{2}^{\frac{n}{2}\left(1-\frac{\beta}{n}\right)}\\
=&C\|u\|_\frac{n}{2}^{\frac{n}{2}\left(1-\frac{\beta}{n}\right)}.
\end{align*}
To estimate $I_{1}$, we first note that, by (\ref{5.100}), we have
\begin{equation*}
[u^\sharp(r)]^\frac{n}{n-2}\leq(1+\varepsilon_2)
[u^\sharp(r)-u^\sharp(R_1)]^\frac{n}{n-2}+C_{\varepsilon_2}[u^\sharp(R_1)]^\frac{n}{n-2}
\end{equation*}
for all $0\leq r\leq R_1$ and $\varepsilon_2>0$. As a consequence, we get
\begin{align}
&\int_{\mathcal{W}_{R_1}}\frac{\Phi(\lambda_n(1-\frac{\beta}{n})(u^\sharp)^\frac{n}{n-2})}
{\left(1+(u^\sharp)^{\frac{n}{n-2}\left(1-\frac{\beta}{n}\right)}\right)F^o(x)^\beta}dx\nonumber\\
=&n\kappa_n \int_0^{R_1} \frac{\Phi(\lambda_n(1-\frac{\beta}{n})[u^*(\kappa_nr^n)]^\frac{n}{n-2})}
{\left(1+[u^*(\kappa_nr^n)]^{\frac{n}{n-2}\left(1-\frac{\beta}{n}\right)}\right)F^o(x)^\beta}dr
\nonumber\\
\leq&  \frac{n\kappa_n}{[u^*(\kappa_nR_1^n)]^{\frac{n}{n-2}(1-\frac{\beta}{n})}}
\int_0^{R_1} \exp(\lambda_n(1-\frac{\beta}{n})[u^*(\kappa_nr^n)]^\frac{n}{n-2})r^{n-1-\beta}dr
\nonumber\\
\leq&
\frac{n\kappa_n\exp(C_{\varepsilon_2}\lambda_n(1-\frac{\beta}{n})[u^*(\kappa_nR_1^n)]^\frac{n}{n-2})}{[u^*(\kappa_nR_1^n)]
^{\frac{n}{n-2}(1-\frac{\beta}{n})}}
\nonumber\\
&\times \int_0^{R_1} \exp(  \lambda_n(1-\frac{\beta}{n}) (1+\varepsilon_2)[u^\sharp(r)-u^\sharp(R_1)]^\frac{n}{n-2} )r^{n-1-\beta}dr
.\label{5.6}
\end{align}
Since
\begin{equation*}
\frac{\lambda_n}{[n\kappa_n^\frac{1}{n}]^\frac{2n}{n-2}}=\left(1-\frac{2}{n}\right)^\frac{n}{n-2},
\end{equation*}
and
\begin{equation*}
0<u^\sharp(r)-u^\sharp(R_1)\leq\frac{1}{[n\kappa_n^\frac{1}{n}]^2}\int_{\kappa_nr^n}^{\kappa_nR_1^n}\frac{f^{**}(s)}{s^{1-\frac{2}{n}}}ds,
\end{equation*}
we have
\begin{align}
&\int_{\mathcal{W}_{R_1}}\frac{\Phi(\lambda_n(1-\frac{\beta}{n})(u^\sharp)^\frac{n}{n-2})}
{\left(1+(u^\sharp)^{\frac{n}{n-2}\left(1-\frac{\beta}{n}\right)}\right)F^o(x)^\beta}dx\nonumber\\
\leq&
\frac{n\kappa_n\exp(C_{\varepsilon_2}\lambda_n(1-\frac{\beta}{n})[u^*(\kappa_nR_1^n)]
^\frac{n}{n-2})}{[u^*(\kappa_nR_1^n)]^{\frac{n}{n-2}(1-\frac{\beta}{n})}}
\nonumber\\
\times &\int_0^{R_1} \exp(  \lambda_n(1-\frac{\beta}{n}) (1+\varepsilon_2)[u^\sharp(r)-u^\sharp(R_1)]^\frac{n}{n-2} )r^{n-1-\beta}dr
\nonumber\\
=&R_1^{n-\beta}\frac{n\kappa_n\exp(C_{\varepsilon_2}\lambda_n(1-\frac{\beta}{n})
[u^*(\kappa_nR_1^n)]^\frac{n}{n-2})}{n[u^*(\kappa_nR_1^n)]^{\frac{n}{n-2}(1-\frac{\beta}{n})}}
\nonumber\\
\times & \int_0^{\infty} \exp([  (1+\varepsilon_2)^{\frac{n-2}{n}}  (1-\frac{2}{n}) (1-\frac{\beta}{n})^{\frac{n-2}{n}}
\int_{\kappa_nR_1^n e^{-t}}^{\kappa_nR_1^n}\frac{f^{**}(s)}{s^{1-\frac{2}{n}}}ds
   ]^{\frac{n}{n-2}})e^{-t(1-\frac{\beta}{n})}dt,
\label{5.6.0}
\end{align}
where the last equation we make the change of variable $r=R_1e^{-\frac{t}{n}}$. Write that
\begin{equation*}
\phi(t)=(\kappa_n)^{\frac{2}{n}}R_1^2(1+\varepsilon_2)^\frac{n-2}{n}(1-\frac{2}{n})f^{\ast\ast}(\kappa_nR_1^ne^{-t})e^{-\frac{2t}{n}}.
\end{equation*}
Then we have
\begin{align*}
&\int_{0}^{\infty}[\phi(t)]^\frac{n}{2}dt\\
=&\kappa_nR_1^n\left(\frac{n-2}{n}\right)^\frac{n}{2}(1+\varepsilon_2)^\frac{n-2}{2}\int_0^{+\infty}[f^{\ast\ast}
(\kappa_nR_1^ne^{-t})]^\frac{n}{2}e^{-t}dt\\
=&\left(\frac{n-2}{n}\right)^\frac{n}{2}(1+\varepsilon_2)^\frac{n-2}{2}\int_0^{\kappa_nR_1^n}[f^{\ast\ast}(s)]^\frac{n}{2}ds\\
\leq&\alpha(1-\varepsilon_0)(1+\varepsilon_2)^\frac{n-2}{2}\left(\frac{n-2}{n}\right)^\frac{n}{2}.
\end{align*}
If we choose $\varepsilon_2=(\frac{1}{1-\varepsilon_0})^\frac{2}{n-2}-1$, we have that
$$
\int_{0}^{\infty}[\phi(t)]^\frac{n}{2}dt\leq 1.
$$
Now, using Lemma \ref{lm4.1} with $\gamma=1-\frac{\beta}{n}$, $p=\frac{n}{2}$ and $a(s,t)=\chi_{[0,t]}(s)$, we have
\begin{equation}\label{5.7}
\int_0^\infty\exp(((1-\frac{\beta}{n})^\frac{n-2}{n}\int_0^t\phi(s)ds)^\frac{n}{n-2})e^{-t(1-\frac{\beta}{n})}dt\leq c_0.
\end{equation}
Noticing that
\begin{align*}
&\int_0^\infty\exp(((1-\frac{\beta}{n})^\frac{n-2}{n}\int_0^t\phi(s)ds)^\frac{n}{n-2})e^{-t(1-\frac{\beta}{n})}dt\\
=&\int_0^\infty\exp(    [(1-\frac{\beta}{n})^\frac{n-2}{n} \int_0^t
(\kappa_n)^{\frac{2}{n}}R_1^2(1+\varepsilon_2)^\frac{n-2}{n}(1-\frac{2}{n})f^{\ast\ast}(\kappa_nR_1^ne^{-s})e^{-\frac{2s}{n}}ds
  ]^{\frac{n}{n-2}})\\
  &   \cdot e^{-t(1-\frac{\beta}{n})}dt\\
=& \int_0^\infty\exp(    [(1-\frac{\beta}{n})^\frac{n-2}{n}
(1+\varepsilon_2)^\frac{n-2}{n}(1-\frac{2}{n})
 \int_{\kappa_nR_1^ne^{-t}}^{\kappa_nR_1^n}
\frac{f^{\ast\ast}(r)}{r^{1-\frac{2}{n}}}dr
  ]^{\frac{n}{n-2}})   e^{-t(1-\frac{\beta}{n})}dt,
\end{align*}
where we make the change of variable $r=\kappa_nR_1^ne^{-s}$, we obtain
\begin{equation}\label{5.7.0}
\int_0^\infty\exp(    [(1-\frac{\beta}{n})^\frac{n-2}{n}
(1+\varepsilon_2)^\frac{n-2}{n}(1-\frac{2}{n})
 \int_{\kappa_nR_1^ne^{-t}}^{\kappa_nR_1^n}
\frac{f^{\ast\ast}(r)}{r^{1-\frac{2}{n}}}dr
  ]^{\frac{n}{n-2}})   e^{-t(1-\frac{\beta}{n})}dt\leq C.
\end{equation}
Since $\varepsilon_2=(\frac{1}{1-\varepsilon_0})^\frac{2}{n-2}-1$, we have $C_{\varepsilon_2}=\varepsilon_0^{-\frac{2}{n-2}}$.
Thus by (\ref{5.6.0}) and (\ref{5.7.0}), we get
\begin{align*}
&\int_{\mathcal{W}_{R_1}}\frac{\Phi(\lambda_n(1-\frac{\beta}{n})(u^\sharp)^\frac{n}{n-2})}{\left(1+(u^\sharp)^{\frac{n}{n-2}\left(1-\frac{\beta}{n}\right)}\right)
F^o(x)^\beta}dx\\
\leq&CR_1^{n-\beta}\frac{\exp(C_{\varepsilon_2}
\lambda_n(1-\frac{\beta}{n})[u^\sharp(R_1)]^\frac{n}{n-2})}{n[u^\sharp(R_1)]^{\frac{n}{n-2}\left(1-\frac{\beta}{n}\right)}}\\
=&CR_1^{n-\beta}\frac{\exp(\varepsilon_0^{-\frac{2}{n-2}}
\lambda_n(1-\frac{\beta}{n})[u^\sharp(R_1)]^\frac{n}{n-2})}{n[u^\sharp(R_1)]^{\frac{n}{n-2}\left(1-\frac{\beta}{n}\right)}}
\end{align*}
Due to $$\int_{\kappa_nR_1^n}^{+\infty}[f^{**}(s)]^\frac{n}{2}ds\leq(\frac{n}{n-2})^\frac{n}{2}\varepsilon_0,$$
by scaling and then by using Lemma \ref{lm3.2}, we have
\begin{equation*}
\left(R_1^n\frac{\exp(\varepsilon_0^{-\frac{2}{n-2}}\lambda_n[u^\sharp(R_1)]^\frac{n}{n-2})}{[u^\sharp(R_1)]^\frac{n}{n-2}}\right)^{1-\frac{\beta}{n}}\leq C\left(\frac{\int_{R_1}^{+\infty}[u^\sharp(R_1)]^\frac{n}{2}r^{n-1}dr}{\varepsilon_0^\frac{n}{n-2}}\right)^{1-\frac{\beta}{n}}.
\end{equation*}
Therefore we obtain
\begin{equation*}
\int_{\mathcal{W}_{R_1}}\frac{\Phi(\lambda_n(1-\frac{\beta}{n})(u^\sharp)^\frac{n}{n-2})}
{\left(1+(u^\sharp)^{\frac{n}{n-2}\left(1-\frac{\beta}{n}\right)}\right)
F^o(x)^\beta}dx\leq C\|u^\sharp\|_\frac{n}{2}^{\frac{n}{2}\left(1-\frac{\beta}{n}\right)}.
\end{equation*}

Finally, we prove that the inequality expressed in Theorem \ref{main-thm-1} is sharp. We need to show that for any $u\in W^{2,\frac{n}{2}}(\mathbb{R}^n)$, $\|\Delta_F u\|_\frac{n}{2}\leq1$,
 the inequality
\begin{equation}\label{5.200}
\int_{\mathbb{R}^n}\frac{\Phi(\lambda_n(1-\frac{\beta}{n}) |u|^\frac{n}{n-2})}{(1+|u|^{\frac{n}{n-2}(1-\frac{\beta}{n})})F^o(x)^\beta}dx\leq C\|u\|_\frac{n}{2}^{\frac{n}{2}(1-\frac{\beta}{n})}
\end{equation}
fails if we replace the growth
\begin{equation*}
\int_{\mathbb{R}^n}\frac{\Phi(\lambda_n(1-\frac{\beta}{n}) |u|^\frac{n}{n-2})}{(1+|u|^{\frac{n}{n-2}(1-\frac{\beta}{n})})F^o(x)^\beta}dx
\end{equation*}
with the higher-order growth
\begin{equation}\label{5.202}
\int\limits_{\mathbb{R}^{n}} \frac{\Phi(\delta |u|^\frac{n}{n-2})}{\left(1 + |u|^{q}\right)F^o(x)^{\beta}}dx,
\end{equation}
where either $\delta>\lambda_n(1-\frac{\beta}{n})$ and $q=\frac{n}{n-2}(1-\frac{\beta}{n})$
or $\delta=\lambda_n(1-\frac{\beta}{n})$ and $q<\frac{n}{n-2}(1-\frac{\beta}{n})$.
In particular, it is sufficient to prove that (\ref{5.200}) fails if we consider an nonlinearity
form (\ref{5.202}) with $\delta=\lambda_n(1-\frac{\beta}{n})$ and $q<\frac{n}{n-2}(1-\frac{\beta}{n})$.

We argue by contradiction. For fixed  $q<\frac{n}{n-2}(1-\frac{\beta}{n})$ and any $u\in W^{2,\frac{n}{2}}(\mathbb{R}^n)$, $\|\Delta_F u\|_\frac{n}{2}\leq1$, we assume that
\begin{equation}\label{5.300}
\int_{\mathbb{R}^n}\frac{\Phi(\lambda_n(1-\frac{\beta}{n}) |u|^\frac{n}{n-2})}{(1+|u|^{q})F^o(x)^\beta}dx\leq C\|u\|_\frac{n}{2}^{\frac{n}{2}(1-\frac{\beta}{n})}.
\end{equation}

We consider the sequence of test functions $u_\epsilon$ similarly introduced by Lu, Tang and Zhu in \cite{LTZ}:
\begin{equation}\label{5.500}
\begin{aligned}
u_{\epsilon}(x) =
\begin{cases}
[\frac{1}{\lambda_n}\ln\frac{1}{\epsilon}]^\frac{n-2}{n}-\frac{F^o(x)^2}
{(\epsilon\ln\frac{1}{\epsilon})^\frac{2}{n}}+\frac{1}{(\ln\frac{1}{\epsilon})^\frac{2}{n}},\quad
&\text{if}\quad F^o(x)\leq\epsilon^\frac{1}{n},\\
n\lambda_n^{\frac{2}{n}-1}[\ln\frac{1}{\epsilon}]^{-\frac{2}{n}}\ln\frac{1}{F^o(x)},\quad&\text{if}\quad\epsilon^\frac{1}{n}\leq F^o(x)\leq1, \\
0,\quad&\text{if}\quad F^o(x)>1,
\end{cases}
\end{aligned}
\end{equation}
where $\epsilon$ is sufficiently small and positive number.
Through the co-area formula (\ref{2.1}) in anisotropic form, we have
\begin{eqnarray}\label{5.210}
\|u_\epsilon\|_\frac{n}{2}^\frac{n}{2}  &= & \int_{\W_{\epsilon^\frac{1}{n}}}|u_\epsilon|^{\frac{n}{2}} dx+\int_{\W_{1}\backslash\W_{\epsilon^\frac{1}{n}}}|u_\epsilon|^{\frac{n}{2}} dx
\nonumber\\
&= & \epsilon O[(\ln\frac{1}{\epsilon})^{\frac{n-2}{2}}]+n(n-2)^{-\frac{n}{2}}
\frac{1}{\ln \frac{1}{\epsilon}}\int_{\epsilon^\frac{1}{n}}^{1}t^{n-1}(\ln \frac{1}{t})^{\frac n2}dt\nonumber\\
&= & O(\frac{1}{\ln\frac{1}{\epsilon}}).
\end{eqnarray}
When $0\leq F^o(x)\leq \epsilon^\frac{1}{n}$, by Lemma \ref{lm-2.1}, we have
\begin{eqnarray*}
\Delta_F u_\epsilon
 & = & \text{div}(F(\nabla u_\epsilon)F_{\xi}(\nabla u_\epsilon))
\nonumber\\
& = &-\frac{1}{(\epsilon\ln\frac{1}{\epsilon})^\frac{2}{n}}
\text{div}(        F(2F^o(x) \nabla F^o(x)  ) F_{\xi}(2F^o(x)\nabla F^o(x)  )      ) \nonumber\\
& = & -\frac{2}{(\epsilon\ln\frac{1}{\epsilon})^\frac{2}{n}}\text{div}(x)\nonumber\\
& = & -\frac{2n}{(\epsilon\ln\frac{1}{\epsilon})^\frac{2}{n}}.
\end{eqnarray*}
Similarly, when $\epsilon^\frac{1}{n}\leq F^o(x)\leq 1$, we have
\begin{eqnarray*}
\Delta_F u_\epsilon
 & = & \text{div}(F(\nabla u_\epsilon)F_{\xi}(\nabla u_\epsilon))
\nonumber\\
& = &-n\lambda_n^{\frac{2}{n}-1}[\ln\frac{1}{\epsilon}]^{-\frac{2}{n}}\text{div}(        F(\frac{\nabla F^o(x)}{F^o(x)}  ) F_{\xi}(\frac{\nabla F^o(x)}{F^o(x)} )      ) \nonumber\\
& = & -n\lambda_n^{\frac{2}{n}-1}[\ln\frac{1}{\epsilon}]^{-\frac{2}{n}}\text{div}(\frac{x}{F^o(x)^2})\nonumber\\
& = & -n(n-2)\lambda_n^{\frac{2}{n}-1}[\ln\frac{1}{\epsilon}]^{-\frac{2}{n}}\frac{1}{F^o(x)^2}.
\end{eqnarray*}
Thus by the co-area formula (\ref{2.1}), we get
\begin{eqnarray*}
\|\Delta_F u_\epsilon\|_\frac{n}{2}^\frac{n}{2}  &= & \int_{\W_{\epsilon^\frac{1}{n}}}|\Delta_F u_\epsilon|^{\frac{n}{2}} dx+\int_{\W_{1}\backslash\W_{\epsilon^\frac{1}{n}}}|\Delta_F u_\epsilon|^{\frac{n}{2}} dx\nonumber\\
&= & \frac{(2n)^{\frac{n}{2}}\kappa_n}{\ln\frac{1}{\epsilon}}+n^{\frac{n}{2}}(n-2)^{\frac{n}{2}}\lambda_n^{1-\frac{n}{2}}
\frac{1}{\ln\frac{1}{\epsilon}}\int_{\epsilon^\frac{1}{n}}^{1}\frac{n\kappa_n r^{n-1}}{r^{n}}dr
\nonumber\\
&= & O(\frac{1}{\ln\frac{1}{\epsilon}})+1.
\end{eqnarray*}
So we can see that
\begin{equation}\label{5.600}
1\leq\|\Delta_F u_\epsilon\|_\frac{n}{2}^\frac{n}{2}\leq1+O(\frac{1}{\ln\frac{1}{\epsilon}}).
\end{equation}
Now let
\begin{equation*}
\widetilde{u}_\epsilon=\frac{u_\epsilon}{\|\Delta_Fu_\epsilon\|_\frac{n}{2}}.
\end{equation*}
It is clear that
\begin{equation*}
\|\widetilde{u}_\epsilon\|_\frac{n}{2}^\frac{n}{2}=O(\frac{1}{\ln\frac{1}{\epsilon}}).
\end{equation*}
By (\ref{5.300}), it follows that
\begin{equation}\label{5.400}
\int\limits_{\mathbb{R}^{n}} \frac{\Phi(\lambda_n(1-\frac{\beta}{n})|\widetilde{u}_\epsilon|^\frac{n}{n-2})}{\left(1 + |\widetilde{u}_\epsilon|^{q}\right)F^o(x)^{\beta}}dx \leq C \|\widetilde{u}_\epsilon\|_\frac{n}{2}^{\frac{n}{2}(1-\frac{\beta}{n})}\leq
C\left(\frac{1}{(\ln\frac{1}{\epsilon})^{^{1-\frac{\beta}{n}}}}\right).
\end{equation}
Also, noting that $u_\epsilon\geq[\frac{1}{\lambda_n}\ln\frac{1}{\epsilon}]^\frac{n-2}{n}$ on $\mathcal{W}_{\epsilon^{\frac{1}{n}}}$, we have that
\begin{equation*}
\begin{aligned}
&\int\limits_{\mathbb{R}^{n}} \frac{\Phi(\lambda_n(1-\frac{\beta}{n})|\widetilde{u}_\epsilon|^\frac{n}{n-2})}{\left(1 + |\widetilde{u}_\epsilon|^{q}\right)
F^o(x)^{\beta}} dx\\
 &\geq \int\limits_{\mathcal{W}_{\epsilon^{\frac{1}{n}}}} \frac{\Phi(\lambda_n(1-\frac{\beta}{n})|\widetilde{u}_\epsilon|^\frac{n}{n-2})}{\left(1 + |\widetilde{u}_\epsilon|^{q}\right)F^o(x)^{\beta}} dx\\
&\geq C \int\limits_{\mathcal{W}_{\epsilon^{\frac{1}{n}}}} \frac{\exp(\lambda_n(1-\frac{\beta}{n})|\widetilde{u}_\epsilon|^\frac{n}{n-2})}{|\widetilde{u}_\epsilon|^{q}F^o(x)^{\beta}} dx\\
&\geq C
\exp\left(\left(1-\frac{\beta}{n}\right)\ln\frac{1}{\epsilon}
\left(\frac{1}{\|\Delta_Fu_\epsilon\|_\frac{n}{2}^\frac{n}{n-2}}\right)\right)\left(\ln\frac{1}{\epsilon}\right)^{-\frac{n-2}{n}q}
\epsilon^{1-\frac{\beta}{n}}.
\end{aligned}
\end{equation*}
Hence we have
\begin{equation*}
\lim_{\epsilon\rightarrow 0} \left(\ln \frac{1}{\epsilon} \right)^{1-\frac{\beta}{n}}\int\limits_{\mathbb{R}^{n}} \frac{\Phi(\lambda_n(1-\frac{\beta}{n}|\widetilde{u}_\epsilon|)}{\left(1 + |\widetilde{u}_\epsilon|^{q}\right)F^o(x)^{\beta}} dx
\geq C\lim_{\epsilon\rightarrow 0}\left(\ln\frac{1}{\epsilon}\right)^{1-\frac{\beta}{n}-\frac{n-2}{n}q}=\infty,
\end{equation*}
which contradicts (\ref{5.400}).
Thus we completed the proof of Theorem \ref{main-thm-1}. \qed
\section{Proof of Theorem \ref{main-thm-2}}
In this section, using the similar trick as the proof of Theorem \ref{main-thm-1},
we prove the singular anisotropic Adams' type inequality on bounded domain $\Omega\subset \mathbb{R}^n$.

{\bf Proof of Theorem \ref{main-thm-2}}:
We set $f=-\Delta_F u$ in $\Omega$. According to Lemma \ref{lm2.2}, we have
\begin{align}\label{7.1}
\int_{0}^{|\Omega|}[f^{\ast\ast}(r)]^\frac{n}{2}dr&\leq\left(\frac{n}{n-2}\right)^\frac{n}{2}\int_{0}^{|\Omega|}[f^\ast(r)]^\frac{n}{2}dr\nonumber \\
&=\left(\frac{n}{n-2}\right)^\frac{n}{2}\|\Delta_F u\|_\frac{n}{2}^\frac{n}{2}\leq\left(\frac{n}{n-2}\right)^\frac{n}{2}.
\end{align}
Then, by Lemma \ref{lm3.1}, we have
\begin{equation*}
0\leq u^\ast(r)\leq\frac{1}{[n\kappa_n^\frac{1}{n}]^2}\int_r^{|\Omega|}\frac{f^{\ast\ast}(\xi)}{\xi^{1-\frac{2}{n}}}d\xi, \text{ for a.e. } 0<r\leq|\Omega|.
\end{equation*}
Hence we can estimate
\begin{eqnarray*}
\int_{\Omega}       \frac{ e^{\lambda_n(1-\frac{\beta}{n}) |u|^\frac{n}{n-2}} }{ F^o(x)^\beta } dx
&=&\int_{\Omega^{\sharp}}\frac{ e^{\lambda_n(1-\frac{\beta}{n})
[u^*(\kappa_n F^o(x)^n )]^\frac{n}{n-2}} }{ F^o(x)^\beta } dx\nonumber\\
&=&\int_{0}^{|\Omega|}\frac{ \kappa_n^{\frac{\beta}{n}}e^{\lambda_n(1-\frac{\beta}{n}) u^*(r)^\frac{n}{n-2}} }{ r^{\frac{\beta}{n}} } dr\nonumber\\
&\leq & \int_0^{|\Omega|}\frac{\kappa_n^{\frac{\beta}{n}}e^{\left(\frac{n-2}{n}(1-\frac{\beta}{n})^{\frac{n-2}{n}}
\int_r^{|\Omega|}\frac{f^{\ast\ast}(\xi)}{\xi^{1-\frac{2}{n}}}d\xi\right)^\frac{n}{n-2}}}{r^{\frac{\beta}{n}}}dr.
\end{eqnarray*}
Making the change of variable $r=|\Omega|e^{-t}$, we have
\begin{align}
&\int_{\Omega}       \frac{ e^{\lambda_n(1-\frac{\beta}{n}) |u|^\frac{n}{n-2}} }{ F^o(x)^\beta } dx\nonumber\\
\leq & |\Omega|^{1-\frac{\beta}{n}}\kappa_n^{\frac{\beta}{n}}\int_0^{\infty}\exp\left\{\left(\frac{n-2}{n}(1-\frac{\beta}{n})^{\frac{n-2}{n}}\int_{|\Omega|e^{-t}}^{|\Omega|}
\frac{f^{\ast\ast}(\xi)}{\xi^{1-\frac{2}{n}}}d\xi\right)^\frac{n}{n-2}-t(1-\frac{\beta}{n})\right\}.
\label{7.2}
\end{align}
Let
\begin{equation*}
\phi(s)=\frac{n-2}{n}|\Omega|^\frac{2}{n}f^{\ast\ast}(|\Omega|e^{-s})e^{-\frac{2s}{n}},\quad \forall s\geq0.
\end{equation*}
It is clear that $\phi(s)\geq0$. By (\ref{7.1}), we obtain
\begin{align*}
\int_0^{+\infty}[\phi(s)]^\frac{n}{2}ds&=\left(\frac{n-2}{n}\right)^\frac{n}{2}|\Omega|\int_0^{+\infty}[f^{\ast\ast}(|\Omega|e^{-s})]^\frac{n}{2}e^{-s}ds \\
&=\left(\frac{n-2}{n}\right)^\frac{n}{2}\int_0^{|\Omega|}[f^{\ast\ast}(r)]^\frac{n}{2}dr\leq1.
\end{align*}
Define
\begin{align*}
F(t)&=(1-\frac{\beta}{n})t-(1-\frac{\beta}{n})\left(\int_0^t\phi(s)ds\right)^\frac{n}{n-2}\\
&=(1-\frac{\beta}{n})t-\left(\frac{n-2}{n}(1-\frac{\beta}{n})^{\frac{n-2}{n}}|\Omega|^\frac{2}{n}\int_0^tf^{\ast\ast}(|\Omega|e^{-s})e^{-\frac{2s}{n}}ds\right)^\frac{n}{n-2}\\
&=(1-\frac{\beta}{n})t-\left(\frac{n-2}{n}(1-\frac{\beta}{n})^{\frac{n-2}{n}}
\int_{|\Omega|e^{-t}}^{|\Omega|}
\frac{f^{\ast\ast}(\xi)}{\xi^{1-\frac{2}{n}}}d\xi\right)^\frac{n}{n-2}.
\end{align*}
Recalling (\ref{7.2}), the proof of Theorem \ref{main-thm-2} reduces to showing that
\begin{equation*}
\int_0^{+\infty}e^{-F(t)}dt\leq c_0.
\end{equation*}
which is nothing but the Lemma \ref{lm4.1}.

Now we use test function in (\ref{5.500}) to prove the sharpness. Let $\widetilde{u}_\epsilon=\frac{u_\epsilon}{\|\Delta_F u_\epsilon\|_\frac{n}{2}}$ and $\delta>\lambda_n(1-\frac{\beta}{n})$. By (\ref{5.210}) and (\ref{5.600}), we know
\begin{equation*}
\|u_\epsilon\|_\frac{n}{2}^\frac{n}{2}=O\left(\frac{1}{\ln\frac{1}{\epsilon}}\right),\quad 1\leq\|\Delta_F u\|_\frac{n}{2}^\frac{n}{2}=1+O\left(\frac{1}{\ln\frac{1}{\epsilon}}\right).
\end{equation*}
Since $u_\epsilon\geq[\frac{1}{\lambda_n}\ln\frac{1}{\epsilon}]^\frac{n-2}{n}$ on $\mathcal{W}_{\epsilon^\frac{1}{n}}$, we have
\begin{align*}
\sup\limits_{u\in W^{2,\frac{n}{2}}(\Omega), \|\Delta_F u\|_\frac{n}{2}\leq1}\int_{\Omega}
\frac{ e^{\delta |u|^\frac{n}{n-2}} }{ F^o(x)^\beta } dx
&\geq\lim\limits_{\epsilon\rightarrow0}\int_{\mathcal{W}_{\epsilon^\frac{1}{n}}}
\frac{e^{\delta \widetilde{u}_\epsilon^\frac{n}{n-2}}}{F^o(x)^\beta}dx\\
&\geq\lim\limits_{\epsilon\rightarrow0} \epsilon\kappa_n\left(e^{\frac{1}{\|\Delta_F u_\epsilon\|_\frac{n}{2}^\frac{n}{n-2}}\frac{\delta}{\lambda_n(1-\frac{\beta}{n})}\ln\frac{1}{\epsilon}}\right)
\\
&=+\infty.
\end{align*}
Thus we completed the proof of Theorem \ref{main-thm-2}. \qed

\end{document}